\documentclass[a4paper]{article}

\usepackage[english]{babel}
\usepackage[utf8x]{inputenc}
\usepackage[T1]{fontenc}

\usepackage[a4paper,top=3cm,bottom=2cm,left=3cm,right=3cm,marginparwidth=1.75cm]{geometry}

\usepackage{amsmath}
\usepackage{graphicx}
\usepackage[colorinlistoftodos]{todonotes}
\usepackage[colorlinks=true, allcolors=blue]{hyperref}
\usepackage{latexsym}
\usepackage{amsmath,amsthm}
\usepackage{amsfonts}
\usepackage{amsgen}
\usepackage{amssymb}
\usepackage{amstext}
\usepackage{amssymb}
\usepackage{dsfont}

\newtheorem{theorem}{Theorem}[section]
\newtheorem{lemma}[theorem]{Lemma}

\newtheorem{remark}[theorem]{Remark}

\DeclareMathOperator{\Dist}{dist}
\DeclareMathOperator{\diam}{diam}

\title{Piecewise contracting maps on the interval: Hausdorff dimension, entropy and attractors}
\author{A.E. Calderón\footnote{Escuela de Ingeniería, Facultad de Ingeniería y Empresa -- Universidad Católica Silva Henríquez, Santiago, Chile. {\tt acalderonc@ucsh.cl}} \ and E. Villar-Sepúlveda\footnote{Department of Engineering Mathematics -- University of Bristol, Bristol, England. {\tt edgardo.villar-sepulveda@bristol.ac.uk}}}

\begin{document}
\maketitle

\begin{abstract}
\noindent We consider the attractor $\Lambda$ of a piecewise contracting map $f$ defined on a compact interval. If $f$ is injective, we show that it is possible to estimate the topological entropy of $f$ (according to Bowen's formula) and the Hausdorff dimension of $\Lambda$ via the complexity associated with the orbits of the system. Specifically, we prove that both numbers are zero.\\

\end{abstract}

\section{Introduction and statements}

A map $f$ on a (non-degenerate) compact interval $X\subset\mathbb R$ is called piecewise contracting interval map (PCIM), if there exists a finite collection $\mathcal Z$ of pairwise disjoint open subintervals of $X$ such that $X=\bigcup_{\scriptscriptstyle Z\in\mathcal Z}\overline{Z}$ and $f|Z:Z\to X$ is a contraction ({with respect to} Euclidean metric) for all $Z\in\mathcal Z$. It suffices to define $f$ on $X^*:=\bigcup_{\scriptscriptstyle Z\in\mathcal Z}Z$ and study the dynamics of $f$ on $\widetilde X:=\bigcap_{j=0}^\infty f^{-j}(X^*)$, the set of points for which all iterates of $f$ are well-defined; therefore, there is no pro\-blem when considering the orbit of points in $\widetilde X$. It is not difficult to see that $\widetilde X$ is dense on $X$ when $f$ is piecewise monotonic (i.e. $f|Z$ is strictly monotone for all $Z\in\mathcal Z$). 

We consider the attractor $\Lambda$ of $f$, which is defined as the asymptotic set that attracts all orbits of points in $\widetilde X$ (a formal definition of $\Lambda$ will be provided in the next section). It is known that the $\omega$-limit set of any point in $\widetilde X$ is non-empty, compact and is contained in $\Lambda$ (see Lemma 2.1 in \cite{CCG21}), but it is not necessarily invariant if it contains points of the finite set $\Delta:=X\setminus X^*$ (as usual, for every $x\in \widetilde X$ we denote by $\omega(x)$ its $\omega$-limit set). For this reason, it is convenient to work with a concept of invariance that does not depend on how $f$ is defined on $\Delta$. We say that a set $A\subset X$ is $f$-pseudo-invariant if, for every $x\in A$,
\[
    \lim_{y\to x^-}f(y)\in A\qquad\text{or}\qquad \lim_{y\to x^+}f(y)\in A.
\]
Note that if $A \subset X$ is $f$-pseudo-invariant, then $A \cap \widetilde X$ is invariant by $f$, i.e. $f(A \cap \widetilde X)\subset A \cap \widetilde X$. As we will see next, it is possible to guarantee that every pseudo-invariant set intersects $\widetilde X$ provided that the collection of one-sided limits of points in $\Delta$ --which we denote by $D$-- is contained in $\widetilde X$. 

\begin{lemma}\label{CAP.TILDE}
    Suppose that $D\subset\widetilde X$. If $A \subset X$ is a non-empty $f$-pseudo-invariant set, then $A\cap\widetilde X$ is a non-empty $f$-invariant set.
\end{lemma}
\begin{proof}
    If $A\subset \widetilde X$, the result follows. Otherwise, let $y\in A\setminus\widetilde X\neq\emptyset$ and consider the smallest integer $t\geq0$ such that $c:=f^t(y)\in A\cap\Delta$. Since $A$ is a pseudo-invariant set, we have that at least one of the one-sided limits of $f$ at $c$ belongs to $A$. Thus, as $D\subset \widetilde X$, we deduce that $A\cap\widetilde X\neq\emptyset$.
\end{proof}

Lemma \ref{CAP.TILDE} is also useful for redefining other concepts in dynamics such as mi\-ni\-ma\-li\-ty. If $D\subset\widetilde X$, we say that a compact and pseudo-invariant set $A\subset X$ is $\widetilde X$-minimal if the orbit of every point in $A\cap\widetilde X$ is dense in $A$. In fact, if $D\subset\widetilde X$, then it is also possible to completely describe the attractor of $f$:

\begin{theorem}[Calder\'on-Catsigeras-Guiraud 2021, \cite{CCG21}]\label{PRINCIPAL}
    If $f$ is piecewise monotonic and $D\subset\widetilde X$, then there exist two natural numbers {$N_1,N_2\geq0$ such that $N_1+N_2\geq1$} and the attractor $\Lambda$ of $f$ can be decomposed as follows:
    \begin{equation}
        \Lambda=\left(\bigcup\limits_{i=1}^{N_1}\mathcal O_i\right)\cup\left(\bigcup\limits_{j=1}^{N_2}K_j\right)\!,\label{DECOMP}
    \end{equation}
    where $\mathcal O_1,\mathcal O_2,\ldots,\mathcal O_{N_1}\subset\widetilde{X}$ are pairwise disjoint periodic orbits and $K_1,K_2,\ldots,K_{N_2}$ are different pseudo-invariant and $\widetilde X$-minimal Cantor sets of $X$. Moreover, for any $x\in\widetilde X$, either there exists $i\in\{1,\dots,N_1\}$ such that $\omega(x)=\mathcal O_i$ or there exists $j\in\{1,\dots,N_2\}$ such that $\omega(x)=K_j$.
\end{theorem}

To establish our main result we will require a global injectivity condition that does not depend on how $f$ is defined on $\Delta$. We say that $f$ satisfies the {\em separation property} if $f$ is piecewise monotonic and $\overline{f(Z)}\cap \overline{f(Z')}=\emptyset$ for all $Z,Z'\in\mathcal Z$ such that $Z\neq Z'$. A PCIM that satisfies the separation property is injective on $X^*$, but not necessarily on the whole set $X$ {(consider the case in which $f$ is not injective on $\Delta$)}. Also, not every injective PCIM on $X^*$ satisfies the separation property. However, every injective PCIM on $X^*$ such that its set $\Delta$ contains only jump-discontinuities satisfies it. Our main result is the following:

\begin{theorem}\label{MAIN}
    Suppose that $f$ satisfies the separation property and $D\subset\widetilde X$. Then, the Hausdorff dimension of $\Lambda$ and the topological entropy of $f|(\Lambda\cap\widetilde X)$ are equal to zero.
\end{theorem}

The concept of topological entropy in Theorem \ref{MAIN} refers to the quantity:
\[
    h_{top}\big(f|(\Lambda\cap\widetilde X)\big)=\lim_{\epsilon\to0}\limsup_{n\to\infty}\frac{\log r_n(\epsilon,\Lambda\cap\widetilde X)}{n},
\]
where $r_n(\epsilon,\Lambda\cap\widetilde X)$ denotes the smallest cardinality of every $(n,\epsilon)$-spanning set for $\Lambda\cap\widetilde X$ with respect to $f$ (see Definition 7.8 in \cite{Wa00}). Although the number $r_n(\epsilon,K)$ is defined for a compact set $K$ in the continuous context, it will be shown that $r_n(\epsilon,\Lambda\cap\widetilde X)$ is well defined for all $n\geq1$ and $\epsilon>0$. Also, note that if $D \subset \widetilde X$ and $f$ is piecewise monotonic, then Theorem \ref{PRINCIPAL} implies that $\Lambda$ is an $f$-pseudo-invariant set. Thus, using Lemma \ref{CAP.TILDE}, we deduce that $\Lambda \cap\widetilde X$ is a non-empty and $f$-invariant set.

{\begin{remark}
    The results presented in Theorem \ref{MAIN} are also proven in \cite{GN22} under the hypotheses of global injectivity and assuming that the map is strictly increasing on each continuity piece. Another difference is that the authors work with the concept of {\it singular entropy}, which was introduced in \cite{MZ92} as an alternative for calculating the entropy according to Bowen's formula for continuous piecewise monotone systems of the interval.
\end{remark}
}

\noindent{\bf Paper organization:} In Section 2 we set the notation that we will use throughout this work. Section 3 is devoted to defining the concept of complexity and its relationship with the {\em atoms} of the system, which allow us to characterize the attractor of $f$. Finally, in Section $4$ we prove Theorem \ref{MAIN}.

\section{Convenient notation}

In what follows, we always assume that $f:X\to X$ is a PCIM and that the topology on $X$ is the one induced by the Euclidean metric. Then, there exist $\lambda\in(0,1)$ and a collection of $N\geq 1$ open non-empty disjoint subintervals $X_1,X_2,\ldots,X_N$ such that $X=\bigcup_{i=1}^N\overline{X_i}$ and
\begin{eqnarray}\label{eq1}
    |f(x)-f(y)|\leq\lambda\,|x-y|\qquad\forall\,x,y\in X_i\,,\;\;\forall\,i\in\{1,2,\ldots,N\}.
\end{eqnarray}
The real number $\lambda$ is called {\em contraction rate} of $f$ and the elements of the collection $\{X_i\}_{i=1}^N$ are called {\em contraction pieces}. We will consider them to be sorted. In particular, let $c_0,c_N$ denote the extreme points of $X$ and $\Delta=\{c_1<c_2<\dots<c_{N-1}\}$ the set of the boundaries of the contraction pieces of $f$; that is,
\[
    X_1=[c_0,c_1)\,,\quad X_2=(c_1,c_2)\,,\quad \dots\;,\quad X_N=(c_{N-1},c_N].
\]
For notational convenience we assume that $X_1$ and $X_N$ are half-closed, but one may also consider the case where one or both pieces are open by adding $c_0$ and/or $c_N$ to $\Delta$. In other words, $\Delta$ must contain all the discontinuity points of the map.

% From \eqref{eq1}, it follows that the points of $\Delta$ are removable (maybe continuity points) or jump discontinuities. Therefore, for every $i\in\{1,\dots,N\}$ the map $f|X_i$ admits a unique continuous extension $f_i:\overline{X_i}\to X$ that satisfies \eqref{eq1} for every pair of points in $\overline{X_i}$. With this, we define the one-sided limits of $f$ at the extreme points of its contraction pieces as
% \[
%     d_0:=f_1(c_0),\ d_{N}:=f_N(c_{N}),\ d_i^{-}:=f_i(c_i)\;\text{ and }\; d_i^+:=f_{i+1}(c_i) \quad \forall\,i\in\{1,\dots,N-1\}.
% \]
% Clearly, $D=\{d_0,d_1^-,\ldots,d_{N-1}^-,d_1^+,\ldots,d_{N-1}^+,d_N\}$. 
As we said before, the attractor $\Lambda$ of $f$ is non-empty, compact, pseudo-invariant and $\widetilde X$-minimal provided that $f$ is piecewise monotonic and $D\subset\widetilde X$. The formal definition of the attractor $\Lambda$ is given by 
\[
    \Lambda:=\bigcap_{n\geq1}\Lambda_n\,, \;\text{ where $\,\Lambda_1=\overline{f(X\setminus\Delta)}\;\,$ and $\;\Lambda_{n+1}=\overline{f(\Lambda_n\setminus\Delta)}\;$ for all $n\geq1$.}
\]
We are interested in computing the topological entropy of $f|(\Lambda\cap\widetilde X)$ and the Hausdorff dimension of $\Lambda$. For the latter, it will be necessary to recall some concepts. If $\delta\geq0$ and $E\subset X$, a $\delta$-cover of $E$ is a countable collection of subsets of $X$ that covers $E$ and {the diameter of each of which is} smaller than or equal to $\delta$. Thus, {for any $E\subset X$, and every $s\geq0$ and $\delta>0$,} we consider the numbers
\[
    \mathcal H_\delta^s(E):=\inf\left\{\sum_{C\in\mathcal C}\big(\!\diam(C)\big)^s\ :\ \text{$\mathcal C$ is a $\delta$-cover of $E$}\right\}
\]
and
\[
\mathcal H^s(E):=\lim_{\delta\to0}\mathcal H_\delta^s(E).
\]
It is known that $\mathcal H^s_\delta$ and $\mathcal H^s$ define an outer measure and a measure on $X$, respectively. Furthermore, $\mathcal H^s$ is called $s$-dimensional Hausdorff measure on $X$. The Hausdorff dimension of a set $E\subset X$, which we denote by $\dim_{\mathcal H}(E)$, is the critical value $s\geq0$ where $\mathcal H^s(E)$ jumps from infinity to zero. This number satisfies --for example-- that $\dim_{\mathcal H}(E)=0$ for every $E\subset X$ such that $\#E\leq\aleph_0$. Moreover, if $s_0:=\dim_{\mathcal H}(E)\geq1$ is a positive integer, then $\mathcal H^{s_0}$ coincides with the Lebesgue measure on $\mathbb R^{s_0}$. If $s_0=0$, then it is known that $E$ is totally disconnected (the reciprocal is not true!). In practical terms, the Hausdorff dimension provides a general notion of the {\em size} of a set in a metric space.

Additionally, we can consider another fractal dimension that allows us to describe the ``size'' of sets. Given $E\subset X$ and $\epsilon>0$, let $\ell(E,\epsilon)$ be the smallest number of intervals with diameter at most $\epsilon$ covering $E$. Then, we can define the {\em box dimension} of $E$ to be 
\[
    \dim_B(E):=\liminf_{\epsilon\to0^+}\frac{\log \ell(E,\epsilon)}{\log(1/\epsilon)}.
\]
Here we use the lower limit to avoid problems with the convergence. Strictly speaking, this is usually called the {\em lower box dimension} and the box dimension is usually said to exist when the limit $\lim_{\epsilon \to 0^+}$ exists. Next, we {give} the existing comparison between the Hausdorff and box dimensions.

\begin{lemma}[Inequality (3.17) in \cite{F04}]\label{COMPARISON}
    For every $E\subset X$, $\dim_{\mathcal H}(E)\leq\dim_B(E)$.
\end{lemma}
{We recall that the previous result is valid in a general metric space, after adapting the definitions accordingly.}

\section{Complexity and atoms}

Let $\mathbb N$ be the set of natural numbers starting at 0. We say that the sequence $\theta=(\theta_t)_{t\geq0}\in\{1,\ldots,N\}^{\mathbb N}$ is the {\em itinerary} of a point $x\in\widetilde X$ if, for every $t\in\mathbb N$ and $i\in\{1,\ldots,N\}$, we have that $\theta_t=i$ if and only if $f^t(x)\in X_i$. Also, the {\em complexity function} of a sequence $\theta=(\theta_t)_{t\geq0}$ is the function $p_\theta(n):=\#L_n(\theta)$ defined for every $n\geq1$, where
\[
    L_n(\theta):=\big\{\theta_t\theta_{t+1}\ldots\theta_{t+n-1}\in\{1,\ldots,N\}^n\ :\ t\geq0\big\}.
\]
Thus, $p_\theta(n)$ gives the number of different words of length $n$ contained in $\theta$. The complexity function of any sequence is a non-decreasing function of $n$. Also, if there exists $n_0\geq1$ such that $p_\theta(n_0+1)=p_\theta(n_0)$, it can be shown that $p_\theta(n)=p_\theta(n_0)$ for all $n\geq n_0$. Theorem \ref{PRINCIPAL} is the result of classifying the orbits of points in $\widetilde X$ according to the complexity associated with their itineraries. % (see \cite{CGM18}).

Note that it is possible to associate the concept of complexity with each point of $\widetilde X$ via its itinerary. It could happen that the complexity associated with orbits of a PCIM grows exponentially. However, when a PCIM satisfies the separation property, the complexity growth is at most affine for each itinerary of the system and $n$ sufficiently large. Specifically, we have the following result:

\begin{theorem}[Catsigeras-Guiraud-Meyroneinc 2018,  \cite{CGM18}]\label{COMPLEXITY.THEOREM}
    Suppose that $f$ satisfies the se\-pa\-ra\-tion property. Let $x\in\widetilde X$ and $\theta\in\{1,\ldots,N\}^{\mathbb N}$ be its itinerary, then there exist $m_0\geq1$, $\alpha\in\{0,\ldots,N-1\}$ and $\beta\in\big\{1,\ldots,1+m_0(N-1-\alpha)\big\}$ such that
    \begin{eqnarray}
        p_\theta(n)=\alpha n+\beta\qquad\forall\,n\geq m_0. \label{EQ.COMPLEXITY}
    \end{eqnarray}
\end{theorem}
Under the assumptions of Theorem \ref{PRINCIPAL} it can be proved that, for every $x\in\widetilde X$, $\omega(x)$ is a periodic orbit if and only if the complexity associated to $x$ is eventually constant; that is, $\alpha=0$ (see Theorem 2.2 in \cite{CCG21}). Furthermore, $\omega(x)$ is a $\widetilde X$-minimal Cantor set if and only if the complexity associated with $x$ is eventually affine with $\alpha \neq 0$ (see Theorem 2.3 in \cite{CCG21}). One of the important tools that allowed the establishment of these results is the so-called {\em atom}.

Let $\mathcal P(X)$ be the power set of $X$. For every $i\in\{1,\ldots,N\}$ let $F_i:\mathcal P(X)\to\mathcal P(X)$ be defined by $F_i(A)=\overline{f(A\cap X_i)}$ for $A\in\mathcal P(X)$. Let $n\geq1$ and $(i_1,\ldots,i_n)\in\{1,\ldots,N\}^n$. We say that
    \[
        A_{i_1i_2\ldots i_n}:= F_{i_n}\circ F_{i_{n-1}}\circ\ldots\circ F_{i_1}(X)
    \]
is an {\em atom of generation $n$} if it is non-empty. We denote $\mathcal A_n$ the set of all atoms of generation $n$. Every atom of generation $n\geq2$ is contained in an atom of previous generation. Precisely, for all $n\geq2$ and $(i_1,\ldots,i_n)\in\{1,\ldots,N\}^n$ we have that 
\[
    A_{i_1\ldots i_n}\subset A_{i_2\ldots i_n}\subset\ldots\subset A_{i_n}.
\]
Also, each atom of any generation is a compact interval contained in $X$. Moreover, note that the attractor of $f$ can be defined in terms of atoms as
\begin{equation}
    \Lambda=\bigcap_{n\geq1}\Lambda_n\,,\quad \text{ where }\;\, \Lambda_n=\bigcup_{A\in\mathcal A_n}A\quad\forall\,n\geq1.\label{ATTRACTOR.ATOMS}
\end{equation}

We are interested in relating the code associated with an atom to the itinerary of points in {$X$}. The following result is relevant for this purpose:

\begin{lemma}[Lemma 2.3 in \cite{CGM18}]\label{ATOMS.LEMMA1}
    Suppose that $f$ satisfies the separation property. Then, for every $n\geq1$,
    \begin{enumerate}
        \item\label{ATOMS.DISJOINT} the collection of atoms of generation $n$ is pairwise disjoint. Precisely, if $A,B\in\mathcal A_n$ are such that $A\cap B\neq\emptyset$, then $A=B$;
        \item if $A_{i_1\ldots i_n},A_{j_1\ldots j_n}\in\mathcal A_n$ are such that $A_{i_1\ldots i_n}=A_{j_1\ldots j_n}$, then $(i_1,\ldots,i_n)=(j_1,\ldots,j_n)$.
    \end{enumerate}
\end{lemma}

\noindent Next, we list some basic properties of atoms. The proof of each property is straightforward and is left as an exercise to the reader. Recall that $\lambda\in(0,1)$ is the contraction rate of $f$.

\begin{lemma}\label{ATOMS.LEMMA}
    Each one of the following statements holds:
    \begin{enumerate}
        \item For all $n\geq1$,
        \[
            \max_{A\in\mathcal A_{n+1}}\diam(A)\leq \lambda \max_{A\in\mathcal A_{n}}\diam(A)\,;
        \]
        \item For all $\epsilon>0$, there exists $n_0\geq1$ such that $\diam(A)<\epsilon$ for every $A\in\mathcal A_{n}$ with $n\geq n_0\,$;
        \item For all $x\in\Lambda$, there exists a decreasing sequence $\{B_k\}_{k\geq1}$ of atoms (i.e. $B_k\supset B_{k+1}$ for all $k\geq 1$) such that $x\in B_k\in\mathcal A_k$ for all $k\geq1$ and
        \begin{equation}
            \bigcap_{k\geq1}B_k=\{x\}\,;\label{ONEPOINT}
        \end{equation}
        \item If $f$ satisfies the separation property, the sequence  $\{B_n\}_{n\geq1}$ defined in item 3 is unique for each $x\in\Lambda$.
    \end{enumerate}
\end{lemma}
\noindent Note that the existence of the sequence of atoms $\{B_k\}_{k \geq 1}$ in Lemma \ref{ATOMS.LEMMA} is guaranteed by \eqref{ATTRACTOR.ATOMS}. Besides, the uniqueness in item 4 of the same Lemma follows directly from Lemma \ref{ATOMS.LEMMA1}. Finally, the relationship between atoms, orbits and itineraries is established in the following result:
\begin{lemma}[Lemmas 2.4 and 2.5 in \cite{CGM18}]\label{ATOMS.LEMMA2}
    Let $x\in\widetilde X$ and $\theta\in\{1,\ldots,N\}^{\mathbb N}$ be its itinerary, then $f^{t+n}(x)\in A_{\theta_t\theta_{t+1}\ldots\theta_{t+n-1}}$ for every $t\geq0$ and $n\geq1$. Moreover, if $f$ also satisfies the separation property and $f^{t+n}(x)\in A_{i_1i_2\ldots i_n}$ for some $t\geq0$ and $n\geq1$, then $(i_1,\ldots,i_n)=(\theta_t,\ldots,\theta_{t+n-1})$.
    \end{lemma}

\begin{remark}
    In a more general setting, atoms also allow defining the attractor of a piecewise contracting map (PCM) on a compact subset of $\mathbb R^k$, $k\geq1$. In \cite{CGMU16}, {the authors use a condition on the growth of the number of atoms of generation $n$ with respect to the contraction rate establishing, in that case, that the attractor of a piecewise contracting map has zero Hausdorff dimension}. Furthermore, an example of a PCM on $\mathbb R^3$ with positive topological entropy is exhibited in the same article. {Moreover, in \cite{G22,GN22,JO19,LN18,LN21}, the authors study the asymptotic dynamics of parametrized fa\-mi\-lies of piecewise affine contractions on the interval and circle using different approaches, without considering the concept of atom. Specifically, the authors of \cite{G22,GN22,JO19,LN18} prove that certain sets of real parameters for which those families admit non-periodic asymptotic dynamics have zero Hausdorff dimension. For the 3-parametric family studied in \cite{LN21}, the computation of the Hausdorff dimension for the set of parameters associated with non-periodic dynamics was not carried out.
    }
\end{remark}

\section{Proof of Theorem \ref{MAIN}}

We say that $C$ is a {\em basic piece} of $\Lambda$ if $C$ is an $\widetilde X$-minimal component of the attractor of $f$; that is, $C\in\{\mathcal O_1,\ldots,\mathcal O_{N_1},K_1,\ldots,K_{N_2}\}$ (see Theorem \ref{PRINCIPAL}). Also, if $C$ is a {basic piece} of $\Lambda$, we denote by $\mathcal A_n(C)$ the set of all atoms of generation $n$ that intersect $C$. Because of \eqref{ATTRACTOR.ATOMS}, for every $n\geq1$, every basic piece is contained in the union of atoms of generation $n$.

\begin{lemma}\label{KSUBSETA}
    Suppose that $f$ satisfies the separation property and $D\subset\widetilde X$. If $K$ is a non-periodic basic piece of $\Lambda$ and $\theta$ the itinerary of a point $x\in K\cap \widetilde X$. Then,
    \[
        \mathcal A_n(K)= \{A_{i_1\ldots i_n}\ :\ i_1\ldots i_n\in L_n(\theta)\}\qquad \forall\,n\geq1.
    \]
    In particular, $\#\mathcal A_n(K) = p_\theta(n)$.
\end{lemma}
\begin{proof}
    Let $n\geq1$. Consider $A\in\mathcal A_n(K)$  and let $y\in A$. By the separation property, there exists $\epsilon>0$ such that $\epsilon<\min\!\big\{\!\Dist(A_1,A_2)\ :\ A_1,A_2\in\mathcal A_n\,\text{ and }\,A_1\neq A_2\big\}$, where $\Dist$ denotes the Euclidean metric on $X$. Next, by the $\widetilde X$-minimality of $K$, there exist $m\geq n$ and $x\in K\cap\widetilde X$ such that
    \[
        |y-f^m(x)|<\min\!\big\{\!\diam(A)/2,\epsilon\big\}.
    \]
    Thus, we deduce that $f^m(x)\in A$. If $\theta$ is the itinerary of $x$, from Lemma \ref{ATOMS.LEMMA2} and item \ref{ATOMS.DISJOINT} of Lemma \ref{ATOMS.LEMMA1} we deduce that 
    \[
        f^{m}(x)=f^{(m-n)+n}(x)\in A_{\theta_{m-n}\ldots\theta_{m-1}}=A.
    \]
    Then, we conclude that $y\in A_{\theta_{m-n}\ldots\theta_{m-1}}$, where $\theta_{m-n}\ldots\theta_{m-1}\in L_n(\theta)$. Thus, we have that 
    \begin{equation}
        \mathcal A_n(K)\subset \{A_{i_1\ldots i_n}\ :\ i_1\ldots i_n\in L_n(\theta)\}. \label{CONTAIN.K}
    \end{equation}
    In addition, the reciprocal contention of \eqref{CONTAIN.K} is clearly valid.
\end{proof}

\begin{remark}\label{PERIODIC}
    Note that Lemma \ref{KSUBSETA} is also valid for periodic basic pieces. In this case, the number of atoms that contain the periodic orbit is eventually constant equal to its period, which also coincides with its complexity.
\end{remark}

\begin{lemma}\label{CORO.SUMA}
    Suppose that $f$ satisfies the separation property and $D\subset\widetilde X$. Then, there exist $k\geq1$ and $x_1,\ldots,x_k\in\Lambda\cap\widetilde X$ such that
    \[
        \Lambda\subset \bigcup_{j=1}^k\left(\bigcup_{w\in L_n(\theta_j)}A_{w}\!\right)\qquad\forall\,n\geq1,
    \]
    where $\theta_j$ is the itinerary of $x_j$ for every $j\in\{1,\ldots,k\}$. In particular, we have that $\#\mathcal A_n=p_{\theta_1}(n)+\ldots+p_{\theta_k}(n)$.
\end{lemma}
\begin{proof}
    By Theorem \ref{PRINCIPAL}, there exists a finite number $k:=N_1+N_2$ of different basic pieces of $\Lambda$, which we sort by $C_1,\ldots,C_k$. For every $i\in\{1,\ldots,k\}$ we choose a point $x_i\in C_i\cap\widetilde X$. If $\theta_i$ is the itinerary of $x_i$ for every $i\in\{1,\ldots,k\}$, from Lemma \ref{KSUBSETA} and Remark \ref{PERIODIC} we obtain the result.
\end{proof}

We want to have a notion of the attractor size of a PCIM. For this purpose, we will compute the Hausdorff dimension of $\Lambda$ using the comparison given in Lemma \ref{COMPARISON}. Also, {we} will need the following elementary inequality:

\begin{lemma}[log-sum inequality, Theorem 2.3 in \cite{KOBA}]\label{LOGSUM}
    Let $k\geq1$ and $a_1,\ldots,a_k$ be non-negative real numbers, then
    \[
        \left(\sum_{i=1}^k a_i\right)\log\!\left(\frac{1}{k}\sum_{i=1}^k a_i\right)\leq \sum_{i=1}^ka_i \log a_i.
    \]
\end{lemma}

The previous lemma will allow us to establish an upper bound for the box dimension and, therefore, for the Hausdorff dimension. With this idea, we can prove the following result:

\begin{theorem}\label{DIMH0}
    Suppose that $f$ satisfies the separation property and $D\subset\widetilde X$. Then, $\dim_{\mathcal H}(\Lambda)$ $=\dim_B(\Lambda)=0$.
\end{theorem}

\begin{proof}
    Thanks to Lemma \ref{COMPARISON}, it is enough to prove that $\dim_B(\Lambda)=0$. By \eqref{ATTRACTOR.ATOMS}, we have that
    \begin{equation}
        \Lambda\subset \bigcup_{A\in\mathcal A_n}A\qquad\forall\,n\geq1.\label{COVER.ATOMS}
    \end{equation}
    Let $n\geq1$ and $\epsilon_n:=\lambda^n\diam(X)$. From item 1 of Lemma \ref{ATOMS.LEMMA} we deduce that $\diam(A)<\epsilon_n$ for every $A\in\mathcal A_n$. Next, from \eqref{COVER.ATOMS} and Lemma \ref{CORO.SUMA}, there exist $x_1,\ldots,x_k\in\Lambda\cap\widetilde X$ such that
    \[
        \ell(\Lambda,\epsilon_n)\leq \#\mathcal A_{n}=\sum_{j=1}^k p_{\theta_j}(n),
    \]
    where $\theta_j$ is the itinerary of $x_j$ for every $j\in\{1,\ldots,k\}$. Now, applying the log-sum inequality to the non-negative numbers $k\,p_{\theta_1}(n),\ldots,k\,p_{\theta_k}(n)$, we obtain
    \[
        \log \ell(\Lambda,\epsilon_n)\leq \log\!\left(\frac{1}{k}\sum_{j=1}^k k\,p_{\theta_j}(n)\!\right)\leq \frac{\sum\limits_{j=1}^k p_{\theta_j}(n)\log \!\big(k\,p_{\theta_j}(n)\big)}{\sum\limits_{j=1}^k p_{\theta_j}(n)}\leq \sum_{j=1}^k \log \!\big(k\,p_{\theta_j}(n)\big).
    \]
    Thus, we deduce that
    \begin{equation}
        \frac{\log \ell(\Lambda,\epsilon_n)}{\log(1/\epsilon_n)}\leq \sum_{j=1}^k\frac{\log \!\big(k\,p_{\theta_j}(n)\big)}{\log(1/\epsilon_n)}=\sum_{j=1}^k\frac{\log \!\big(k\,p_{\theta_j}(n)\big)}{n\log(1/\lambda)+\log(1/\diam(X))}.
        \label{BOX.INEQ}
    \end{equation}
    By Theorem \ref{COMPLEXITY.THEOREM}, we have that $p_{\theta_j}(n)$ is at most an affine function for $n$ sufficiently large. Therefore, 
    \[
        \frac{\log\!\big(k\,p_{\theta_j}(n)\big)}{n\log(1/\lambda)+\log(1/\diam(X))}\;\stackrel{n\to\infty}\longrightarrow\; 0\qquad\forall\,j\in\{1,\ldots,k\}.
    \]
    Thus, from \eqref{BOX.INEQ} we deduce that {$\liminf_{\epsilon \to 0^+} \frac{\log \ell(\Lambda,\epsilon)}{\log(1/\epsilon)}\leq 0$}, concluding that {$\dim_B(\Lambda)=0$}. Thus, by Lemma \ref{COMPARISON}, we have that {$\dim_{\mathcal H}(\Lambda)=0$}.
\end{proof}

Theorem \ref{DIMH0} corresponds to the first part of Theorem \ref{MAIN}. Now, we are going to prove the second and final part. Recall that $\Lambda\cap\widetilde X$ is a non-empty and $f$-invariant set whenever $f$ is piecewise monotonic and $D\subset\widetilde X$. Given $n\geq1$ and $\epsilon>0$, we say that a set {$E\subset \Lambda\cap\widetilde X$} is an $(n,\epsilon)$-{\em spanning set} for $\Lambda\cap\widetilde X$ if, for every $x\in\Lambda\cap\widetilde X$, there exists $y\in E$ such that
\[
    \big|f^j(x)-f^j(y)\big|<\epsilon\qquad\forall\,j\in\{0,\ldots,n-1\}.
\]
Thus, we define $r_n(\epsilon,\Lambda\cap\widetilde X)$ as the smallest cardinality of every $(n,\epsilon)$-spanning set for $\Lambda\cap\widetilde X$ with respect to $f$. The number $r_n(\epsilon,K)$ is well defined when $K$ is a compact set over which $f$ is continuous; however, this is not the usual case for PCIMs.

\begin{lemma}\label{WELL}
    Suppose that $f$ is piecewise monotonic and $D\subset\widetilde X$. Then, for every $n\geq1$ and $\epsilon>0$, $r_n(\epsilon,\Lambda\cap\widetilde X)$ is a finite number.
\end{lemma}
\begin{proof}
Let $n\geq1$ and $\epsilon>0$. From item 2 of Lemma \ref{ATOMS.LEMMA}, there exists $n_0\geq n$ such that
    \begin{equation}
        \diam A<\epsilon\qquad \forall\,A\in\mathcal A_{n_0}. \label{ATOM.EPSILON}
    \end{equation}
    Next, denote by $\mathcal C(\mathcal A_{n_0},n)$ the collection of all connected components of
    \[
        A\setminus\big(\Delta\cup f^{-1}(\Delta)\cup\ldots\cup f^{-n+1}(\Delta)\big)\quad\text{with $A \in \mathcal A_{n_0}$.}
    \]
{Since $f$ is piecewise monotonic, the collection $\mathcal C(\mathcal A_{n_0},n)$ is finite. Furthermore, from \eqref{ATTRACTOR.ATOMS} we can consider the subcollection $\mathcal C^*(\mathcal A_{n_0},n)\subset\mathcal C(\mathcal A_{n_0},n)$ of connected components that intersect $\Lambda\cap\widetilde X$, whose union covers said set. Thus, for every $B\in \mathcal C^*(\mathcal A_{n_0},n)$ we can take $x_B\in B\cap\Lambda\cap\widetilde X\neq\emptyset$. It is clear that $E_{n,\epsilon}:=\{x_B\ :\ B\in\mathcal C^*(\mathcal A_{n_0},n)\}\subset\Lambda\cap\widetilde X$ is a finite $(n,\epsilon)$-spanning set. This implies that $r_n(\epsilon,\Lambda\cap\widetilde X)<\infty$, concluding the proof.}
\end{proof}

Next, we establish and prove the last part of Theorem \ref{MAIN}.

\begin{theorem}\label{BOWEN}
    If $f$ satisfies the separation property and $D\subset\widetilde X$, then $h_{top}\big(f|(\Lambda\cap\widetilde X)\big)=0$.
\end{theorem}
\begin{proof}
    Let $n\geq1$ and $\epsilon>0$ be such that 
    \[
        \epsilon<\min\!\big\{|c-c'|\ :\ c,c'\in\Delta \,\text{ such that } c\neq c' \big\}.
    \]
    Consider $n_0\geq n$ such that \eqref{ATOM.EPSILON} holds. Note that every atom of generation $n_0$ or higher has at most one element of $\Delta$. Since $f$ satisfies the {separation property we deduce, from \eqref{EQ.COMPLEXITY} and Lemma \ref{CORO.SUMA},} that  
    \[
        \#\mathcal A_{n+n_0}-\#\mathcal A_{n_0}=\sum_{m=n_0}^{n_0+n-1}\big(\#\mathcal A_{m+1}-\#\mathcal A_{m}\big)\leq n(N-1). 
    \]
    Next, noting that $\#\mathcal C(\mathcal A_{n_0},n)=\#\mathcal A_{n+n_0}$, we have the following inequalities:
    \[
        r_n(\epsilon,\Lambda\cap\widetilde X) \leq \#\mathcal C(\mathcal A_{n_0},n)\leq n(N-1)+\#\mathcal A_{n_0}.
    \]
    Thus, we obtain that
    \[
        h_{top}\big(f|(\Lambda\cap\widetilde X)\big)\leq \lim_{\epsilon\to0}\limsup_{n\to\infty}\frac{\log \big(n(N-1)+\#\mathcal A_{n_0}\big)}{n}=0,
    \]
    which concludes the proof.
\end{proof}

\begin{remark}
    From Lemma \ref{CORO.SUMA} we know that the number of atoms of generation $n\geq1$ is directly related to the complexity of the system. In fact, it can be said that the eventual affine linear growth of the complexity associated with $f$ allows us to prove {Theorem} \ref{BOWEN}. If $f$ is not injective, the relationship between topological entropy and the complexity is not clear. Still, it is common to define the topological entropy of piecewise continuous maps of the interval using the complexity associated with the orbits of the system.
\end{remark}

\noindent{\bf Acknowledgments:} A.E.C. was supported by ANID Fondecyt Iniciación N$^{\circ}$11230064, ANID Fondecyt Re\-gu\-lar N$^{\circ}$1230569, MathAmsud Project VOS 22-MATH-08 and MathAmsud Project TOMCAT 22-MATH-10. E.V-S. was supported by Ph.D. funding from ANID, Beca Chile Doctorado en el Extranjero, number 72210071. {Furthermore, we sincerely thank P. Guiraud and E. Ugalde for their va\-lua\-ble support and guidance in the initial ideas of this study. Finally, we would like to thank the anonymous referee for the careful review of this manuscript; their comments greatly improved our work.}

\end{document}